\documentclass[12pt]{article}

\usepackage{amsmath}
\usepackage{amssymb}
\usepackage{amsthm}

\newtheorem{theorem}{Theorem}[section]
\newtheorem{lemma}[theorem]{Lemma}
\newtheorem{corollary}[theorem]{Corollary}
\newtheorem{proposition}[theorem]{Proposition}

\newtheorem{remark}{Remark}[section]
\newtheorem{definition}{Definition}[section]

\numberwithin{equation}{section}

\newcommand{\R}{\mathbb{R}}
\newcommand{\C}{\mathbb{C}}
\newcommand{\Z}{\mathbb{Z}}
\newcommand{\N}{\mathbb{N}}

\newcommand{\ov}[1]{\overline{#1}}
\newcommand{\cc}[1]{\widetilde{#1}}
\newcommand{\set}[2]{\left\{#1:#2\right\}}
\newcommand{\Span}[2]{\mathop{\mathrm{span}}\set{#1}{#2}}
\newcommand{\ip}[2]{\left\langle #1,#2\right\rangle}
\newcommand{\Zp}{\Z_{\ge0}}
\newcommand{\ssc}[2]{\mathstrut^{#1}\!#2}
\newcommand{\ind}[2]{\mathstrut^{#1}\!#2}
\newcommand{\te}[1]{\breve{#1}}
\newcommand{\rank}{\mathop{\mathrm{rank}}}
\newcommand{\Id}{\mathrm{I}}
\newcommand{\dotpr}{\mathbin{\cdot}}

\newcommand{\lxl}{\mathrel{<_{\mathrm{lex}}}}
\newcommand{\lxg}{\mathrel{>_{\mathrm{lex}}}}
\newcommand{\lxge}{\mathrel{\ge_{\mathrm{lex}}}}

\newcommand{\cA}{\mathcal{A}}

\newcommand{\cD}{\mathcal{D}}
\newcommand{\cP}{\mathcal{P}}
\newcommand{\cV}{\mathcal{V}}

\newcommand{\bD}{\mathbf{D}}
\newcommand{\bA}{\mathbf{A}}
\newcommand{\bI}{\mathbf{I}}
\newcommand{\bG}{\mathbf{G}}
\newcommand{\bS}{\mathbf{S}}

\newcommand{\Ft}{\mathfrak{F}}
\newcommand{\Pj}{\mathfrak{P}}

\newcommand{\LL}{{l}}
\newcommand{\K}{{k}}
\newcommand{\kk}{{r}}
\newcommand{\jj}{{q}}
\newcommand{\NN}{{n}}
\newcommand{\lL}{{m}}
\newcommand{\mm}{{j}}

\begin{document}

\title{Matrix method of polynomial solutions to constant coefficient PDE's}

\author{Victor G.~Zakharov\\[0.5ex]
Institute of Continuous Media Mechanics UB RAS,\\
Perm, 614013, Russia\\
\tt victor@icmm.ru}

\maketitle{}

\begin{abstract}
In the paper, we introduce a matrix method to constructively determine spaces of polynomial
solutions (in general, multiplied by exponentials) to a system of constant coefficient linear PDE's
with polynomial (multiplied by exponentials) right-hand sides.
The matrix method reduces the funding of a polynomial subspace of null-space of a differential
operator to the funding of the null-space of a block matrix.
Using this matrix approach, we investigate some linear algebra properties of the spaces of polynomial solutions.
In particular, for a solution space containing polynomials up to some arbitrarily large degree, we can
determine dimension and basis of the space. Some examples of polynomial (multiplied by exponential, in general)
solutions of the Laplace, Helmholtz, Poisson equations are considered.
\end{abstract}

\medskip

\\[1eX]

\section{Introduction}

The polynomial solutions to linear constant coefficient PDE's is the well-known problem of algebra,
see, for example, \cite{AbramovPetkovsek,Horvath,Pedersen,Reznick,Smith}. 
However, under the  Fourier transform, a linear constant coefficient PDE is dual
to an algebraic polynomial. Thus the problem to find a polynomial solution to linear constant coefficient PDE's
can be transformed to solve a linear algebraic system. So, the linear algebra approach allows to solve
nonhomogeneous and induced by nonhomogeneous polynomials PDE's. Hence the matrix method allows to find
exponential solutions (in general, multiplied by algebraic polynomials) and to solve to PDE's with
polynomial (multiplied by exponentials) right-hand sides. Moreover, the matrix method enables to determine
some (linear algebra) characteristics such as dimension, basis, affinely-invariance, maximal total degrees 
of polynomials, etc., of a solution space. Note that the matrix method is valid for polynomials that induce
PDE's with coefficients from any algebraically closed field. Note also that the  matrix methods can be easily 
algorithmized.

In addition, the matrix approach can be generalized to find polynomial solutions to PDE's with polynomial
coefficients. This will be discussed elsewhere. 

Note that the discussed in this paper matrix approach to solve linear constant coefficient PDE's
was stimulated by a generalization of the Strang-Fix conditions, see~\cite{deBoor,DahmenMicchelli}. 
In particular, the polynomial
(multiplied by exponentials) solutions to the well-known differential equations (like Laplace's equation),
when we take a root of the operator symbol that the root is not the origin, were obtained, see~\cite{Zakh20}.

The paper is organized as follows. Section~\ref{SectionNotationsDefinitions} 
contains used in the paper notations and definitions.
In particular, in Subsection~\ref{SubSectionOrderedSets}, 
the lexicographically ordered sets of polynomials and derivatives are introduced.
Section~\ref{SectionMatrixLinearSystem} is devoted to the matrix of the linear system; 
in Subsections~\ref{SubsectionFormationMatrix}, 
a method to
construct the matrix is presented, and, in Subsection~\ref{SubsectionSomeProperties}, 
some properties of the matrix are discussed.
In Section~\ref{SectionSolutionMethods}, the matrix method to solve (in particular, nonhomogeneous and induced by 
nonhomogeneous polynomials) 
PDE's are discussed. Section~\ref{SectionExamples}, is devoted to examples to find polynomial solutions
to some PDE's.


\section{Notations and definitions}\label{SectionNotationsDefinitions}

\subsection{Basic notations}


Let  $\delta_{ij}:=
     \left\{
       \begin{aligned}
         &0&\mbox{ if }&i\ne j ,\\
         &1&\mbox{ if }&i=j
       \end{aligned}\right.$
be the Kronecker delta and $\delta$ be the Dirac delta-distribution.

A {\em multi-index} $\alpha$ is a $d$-tuple
$(\alpha_1,\dots,\alpha_d)$ with its components being nonnegative
integers, i.\,e., $\alpha\in\Z^d_{\ge0}$.
The {\em length} of a multi-index $\alpha:=(\alpha_1,\dots,\alpha_d)\in\Z^d_{\ge0}$ is defined as $\alpha_1+\cdots+\alpha_d$
and denoted by $|\alpha|$.
For
$\alpha=(\alpha_1,\dots,\alpha_d)$,
$\beta=(\beta_1,\dots,\beta_d)$, we write $\beta\le\alpha$ 
if $\beta_j\le\alpha_j$ for all $j=1,\dots,d$.
The {\em factorial} of $\alpha$ is $\alpha! :=
\alpha_1!\cdots\alpha_d!$. The {\em binomial coefficient} for
multi-indices $\alpha,\beta$ is
\begin{equation*}
   \dbinom{\alpha}{\beta}:=\dbinom{\alpha_1}{\beta_1}\cdots\dbinom{\alpha_d}{\beta_d}
   =\frac{\alpha!}{\beta!(\alpha-\beta)!}.
\end{equation*}
By definition, put
\begin{equation}\label{BinomialCoefficientVanishes}
  \dbinom{\alpha}{\beta}=0\qquad \mbox{if $\beta\nleq\alpha$}.
\end{equation}

By $x^\alpha$, where $x=(x_1,\dots,x_d)$, 
$\alpha=(\alpha_1,\dots,\alpha_d)\in\Z^d_{\ge0}$, denote a
monomial $x_1^{\alpha_1}\cdots x_d^{\alpha_d}$. Note that the {\em
total degree} of $x^\alpha$ is $|\alpha|$.
By $\Pi_\LL$, $\LL\in\Z_{\ge0}$, denote the space of (homogeneous)
polynomials that
the total degrees of the polynomials are
equal to $\LL$: $\Pi_\LL:=\Span{x^\alpha}{
\alpha\in\Z^d_{\ge0}, |\alpha|= \LL}$; and by $\Pi_{\le \LL}$ denote
the space of 
polynomials that
the total degrees of the polynomials are less
than or equal to $\LL$: $\Pi_{\le \LL}:=\Span{x^\alpha}{
\alpha\in\Z^d_{\ge0}, |\alpha|\le \LL}$ (here and in the sequel,
`$\mathop{\mathrm{span}}$' means the linear span over $\C$).
\begin{remark}
Since the linear algebra definitions and assertions are valid for {\em any} field; we can consider
polynomials with coefficients from an arbitrary field.

On the other hand, we would like to use some algebraically closed fields (a field $\C$, for example)
or we must use algebraic extensions of fields.
%
\end{remark}

The dot product of two vectors ($d$-tuples) $x=(x_1,\dots,x_d)$, $y=(y_1,\dots,y_d)$ is
$x\dotpr y:=x_1y_1+\cdots+x_d y_d$.
If all the polynomials from the space $\Pi_\LL$ 
multiplied by
an exponential $e^{ix_0\dotpr x}$, where $x_0\in\C^d$ is a given point; then we shall write $e^{ix_0\dotpr x}\Pi_\LL$
(for $\Pi_{\le \LL}$, 
$e^{ix_0\dotpr x}\Pi_{\le \LL}$).

Let $D^\alpha$ imply a differential operator
$D_1^{\alpha_1}\cdots D_d^{\alpha_d}$, where $D_n$, $n=1,\dots,d$,
is the partial derivative with respect to the $n$th coordinate.
Note that $D^{(0,\dots,0)}$ is the identity operator.
Abusing notations, for a function $f=f(x)$ and {\em constant} point $x_0$ we
shall write everywhere $D^\alpha f(x_0)$, meaning, in fact, $\left.D^\alpha f(x)\right|_{x=x_0}$.

The multi-dimensional
version of the {\em Leibniz rule} is
\begin{equation}\label{MultidimensionalLeibnizRule}
  (fg)^{(\alpha)}=\sum_{\begin{subarray}{l}
             \beta\in\Z^d_{\ge0}\\ \beta\le\alpha
           \end{subarray}}
     \dbinom{\alpha}{\beta}f^{(\beta)} g^{(\alpha-\beta)},\quad
                  \alpha\in\Z_{\ge0}^d,  
\end{equation}
where the functions $f(x)$, $g(x)$, $x=\left(x_1,\dots,x_d\right)$, are
sufficiently differentiable.

The Fourier
transform of a function $f\in L^1\left(\R^d\right)$ is defined by
\begin{equation*}
  f(x)\mapsto\hat f(\xi)=\left(\Ft f\right)(\xi):=(2\pi)^{-d/2}\int_{\R^d} f(x)e^{-i\xi\dotpr x}\,dx,\qquad \xi\in\R^d.
\end{equation*}
Note that the Fourier transform can be extended to
compactly supported functions (distributions)
that the functions belong, for example, to the {\em space of
tempered distributions} $S'\left(\R^d\right)$.
Moreover, the domain 
of the Fourier transform 
can be extended (it is possible, in particular, for a compactly supported function) to the whole complex space $\C^d$.
And the Schwartz space $S$ (of test functions), i.\,e., the space of functions that all the derivatives of the functions
are rapidly decreasing,
also can be extended to $\C^d$.  

So, for $x_0\in\C^d,\alpha\in\Z^d_{\ge0}$, we have the following formula
\begin{equation*}
  \left(\Ft e^{ix_0\dotpr x}x^\alpha\right)(\xi)=i^{|\alpha|} D^\alpha\delta(\xi-x_0),
          \qquad \xi\in\C^d.
\end{equation*}

\begin{definition}\label{InnerProductDefinition}
Let $f,g\in L^2\left(\C^d\right)$ be complex functions. 
Then an
inner product in the space $L^2\left(\C^d\right)$ is
\begin{equation}\label{InnerProduct}
  \ip{f}{g}:=\int_{\C^d}f(x)\ov{g(x)}\,dx.
\end{equation}
Here and in the sequel, the overline $\ov{\ \cdot\ }$ denotes the complex conjugation. 
\end{definition}

In the following definition, we consider complex distributions and complex test
functions, see for example~\cite{GelfandShilov,KolmogorovFomin}.

\begin{definition}\label{ComplexValuedDistributionDefinition}
Let $\phi\in S(\C^d)$ be a complex test function.
Let $f=f(x)$, $x\in \C^d$, be a locally integrable on $\C^d$ complex function.
Then the function $f$ induces some distribution
(continuous linear functional)
on $S(\C^d)$ as follows
\begin{equation}\label{ComplexValuedDistribution}
   T_f(\phi):=\int_{\C^d}\ov{f(x)}\phi(x)\,dx=\ov{\ip{f}{\phi}},
\end{equation}
where $\ip{\dotpr}{\dotpr}$, in the right-hand side of~\eqref{ComplexValuedDistribution},
is the inner product defined by~\eqref{InnerProduct}.
\end{definition}

\begin{remark}
Any functional defined by~\eqref{ComplexValuedDistribution} is a {\em linear functional};
in particular, the functional is {\em homogeneous}:
$$
  T_f(a\phi)=a T_f(\phi),
$$
where $a$ is a complex valued function.
\end{remark}



In the paper, we usually denote matrices by upper-case bold symbols or enclose the symbols of matrices
in the square brackets. On the other hand, abusing notation slightly,
we shall denote a vector of some linear space by plain lower-case symbol and interpret the vector of a
linear space as a column vector.
\begin{definition}
Suppose ${\bA}:=[a_{ij}]$ $(1\le i\le n, 1\le j\le m)$, $a_{ij}\in\C$, is an $n\times m$ matrix.
By definition, put
\begin{equation*}
  \ker{\bA}:=\set{v\in\C^{m}}{{\bA}v=0}.
\end{equation*}
We say that the linear space $\ker{\bA}$ is the (right) {\em
null-space} of the matrix $\bA$.
\end{definition}

\begin{remark}
However sometimes we shall treat 
a null-space $\ker{\bA}$ as ($\dim\ker{\bA}$)-column matrix of vectors that forms a basis for $\ker{\bA}$.
\end{remark}

Now recall some block matrix notions.
A {\em
block matrix} is a matrix broken into sections called {\em blocks}
or {\em submatrices}. A {\em block diagonal matrix} is a block
matrix 
such that the main diagonal square
submatrices can be non-zero and all the off-diagonal submatrices are
zero matrices. The (block) diagonals can be specified by an index
$k$ measured relative to the main diagonal, thus the main diagonal
has $k=0$ and the $k$-diagonal consists of the entries on the
$k$th diagonal above the main diagonal. Note that all the
$k$-diagonal submatrices, except submatrices on the main
diagonal, can be non-square. 

\subsection{Ordered sets}\label{SubSectionOrderedSets}

By $\lxl$ we denote some {\em lexicographical order} and
by ${\cA}_{\K}$, ${\K}\in\Z_{\ge0}$, denote the {\em
lexicographically ordered set} of all multi-indices of length
${\K}$
\begin{equation*}
 {\cA}_{\K}:=\left(\mathstrut^1\!\alpha,\mathstrut^2\!\alpha,\dots,
           \mathstrut^{d({\K})}\!\alpha\right),\qquad
  \begin{aligned}
           &\ind{{\jj}}{\alpha}\in\Z^d_{\ge0},\ |\ind{{\jj}}{\alpha}|={\K},\ {\jj}=1,\dots,d({\K}),\\
           &\ind{{\jj}}{\alpha}\lxl\ind{{\jj}'}{\alpha}\ \Longleftrightarrow\ {\jj}<{\jj}',
  \end{aligned}
\end{equation*}
where
\begin{equation*}
    d({\K}):=\dbinom{d+{\K}-1}{{\K}}
          =\frac{(d+{\K}-1)!}{{\K}!(d-1)!}
\end{equation*}
is the number of ${\K}$-combinations with repetition from the $d$
elements.

By $\cc{\cA}_{\K}$ we denote 
a {\em concatenated set of
multi-indices}
$$
  \cc{\cA}_{\K}:=\left(\cA_0,\cA_1,\dots,\cA_{\K}\right),
$$
where the comma must be considered as a concatenation
operator to join 2 sets. Actually the order of 
$\cc{\cA}_{\K}$ is the {\em graded lexicographical order.} By $\cc d({\K})$ denote
the length of a concatenated set like $\cc{\cA}_{\K}$
\begin{equation*}
  \cc d({\K}):=d(0)+d(1)+\cdots+d({\K})=\dfrac{(d+{\K})!}{{\K}!d!}.
\end{equation*}

By $\cP_{\K}$, ${\K}\in\Z_{\ge0}$, denote the lexicographically
ordered set of all monomials of total degree ${\K}$
\begin{equation*}
  \cP_{\K}(x) := \left( x^{\ind{1}{\alpha}},\dots, x^{\ind{d({\K})}{\alpha}}\right),\qquad
         x=\left(x_1,\dots,x_d\right),\
                    \left(\mathstrut^1\!\alpha,\dots,\mathstrut^{d({\K})}\!\alpha\right)={\cA}_{\K}.
\end{equation*}
For $\beta\in\Z^d_{\ge0}$, by ${\cP}^\beta_{\K}$ denote the
following set of monomials
\begin{equation}\label{P(x)}
  {\cP}^\beta_{\K}(x):=\left(\dbinom{\ind{1}{\alpha}}{\beta}x^{\ind{1}{\alpha}-\beta},\dots,
     \dbinom{\ind{d({\K})}{\alpha}}{\beta}x^{\ind{d({\K})}{\alpha}-\beta}\right),\qquad
         \left(\mathstrut^1\!\alpha,\dots,\mathstrut^{d({\K})}\!\alpha\right)={\cA}_{\K}.
\end{equation}
Similarly, define the ordered sets of differential operators as
\begin{align}
  \cD_{\K} &:= \left((-i)^{{\K}}D^{\ind{1}{\alpha}},\dots,(-i)^{{\K}}D^{\ind{d({\K})}{\alpha}}\right),
          \nonumber\\
  {\cD}^\beta_{\K} &:= \left((-i)^{{\K}-|\beta|}\dbinom{\ind{1}{\alpha}}{\beta}D^{\ind{1}{\alpha}-\beta},\dots,
              (-i)^{{\K}-|\beta|}\dbinom{\ind{d({\K})}{\alpha}}{\beta}D^{\ind{d({\K})}{\alpha}-\beta}\right),
           \label{DBetaVector}
\end{align}
where
$\left(\mathstrut^1\!\alpha,\dots,\mathstrut^{d({\K})}\!\alpha\right)={\cA}_{\K}$. 
Note that if $\beta\not\le\ind{{\jj}}{\alpha}$, then the ${\jj}$th
entries of~\eqref{P(x)} and~\eqref{DBetaVector} are zeros.
Moreover, if $|\beta|> {\K}$; then sets~\eqref{P(x)},~\eqref{DBetaVector}
are zero sets.

By $\cc\cP_{\K}$ and $\cc\cP^\beta_{\K}$ denote the following concatenated sets of
monomials

\begin{align}
    &\cc{\cP}_{\K} :=
    \left(\cP_0,\cP_1,\dots,\cP_{\K}\right),
    \label{ConcatenatedP(x)}\\
    &{\cc{\cP}}^\beta_{\K}:=\left(\cP_0^\beta,\cP_1^\beta,\dots,\cP_{\K}^\beta\right).
    \label{ConcatenatedP(x)beta}
\end{align}
The concatenated sets of derivatives 
are defined similarly to~\eqref{ConcatenatedP(x)}, \eqref{ConcatenatedP(x)beta}
\begin{align}
  \label{ConcatenatedD(x)}
  &\cc{\cD}_{\K} :=
  \left(\cD_0,\cD_1,\dots,\cD_{\K}\right),\\
  &{\cc{\cD}}^\beta_{\K}:=\left(\cD_0^\beta,\cD_1^\beta,\dots,\cD_{\K}^\beta\right).\nonumber
\end{align}


\section{The matrix of the linear system}\label{SectionMatrixLinearSystem}

As it has been said, 
we shall frequently interpret the ordered sets (for example, 
${\cP}_{\K}$, ${\cD}_{\K}$) as {\em
row vectors} and enclose their symbols in the square brackets
($\left[{\cP}_{\K}\right]$, $\left[{\cD}_{\K}\right]$).

\subsection{Formation of the matrix}\label{SubsectionFormationMatrix}


For some ${\K},\LL\in\Z_{\ge0}$, ${\K}\le \LL$, define a $\cc d(\LL)\times d({\K})$ matrix $\bD_{\K}$ as follows
\begin{equation}\label{MatrixD}
  \bD_{\K}
  :=\begin{bmatrix}
      \bD^0_{\K} \\[0.5ex] \bD^1_{\K} \\ \vdots \\ \bD_{\K}^{\K} \\ 0 \\ \vdots \\ 0 \\
    \end{bmatrix},
\end{equation}
where $\bD_{\K}^{\kk}$ are $d({\kk})\times d({\K})$, ${\kk}=0,1,\dots {\K}$, submatrices
defined as 
\begin{equation}\label{MatrixDlL}
  \bD_{\K}^{\kk}
  :=
  \begin{bmatrix}
    \left[\cD^{\ind{1}{\beta}}_{\K}\right] \\[1.5ex] \left[\cD^{\ind{2}{\beta}}_{\K}\right] \\ \vdots \\
    \left[\cD^{\ind{d({\kk})}{\beta}}_{\K}\right]  \\[1.5ex]
  \end{bmatrix},\qquad
  \left(\mathstrut^1\!\beta,\dots,\mathstrut^{d({\kk})}\!\beta\right)={\cA}_{\kk},
\end{equation}
and the row vectors $\left[\cD^{\ind{{\jj}}{\beta}}_{\K}\right]$, ${\jj}=1,\dots,d({\kk})$, are
given by~\eqref{DBetaVector}. By definition, if ${\kk}>{\K}$;
then $\bD^{\kk}_{\K}$ is a zero matrix.

Finally, for $\LL\in\Zp$, define a $\cc d(\LL)\times\cc d(\LL)$ matrix $\cc{\bD}_\LL$ as
\begin{equation}\label{ExtendedMatrixD}
  \cc{\bD}_\LL:=\begin{bmatrix}
      \bD_0 & \bD_1 & \dots & \bD_{\LL-1} & \bD_\LL\\
     \end{bmatrix}
  =\begin{bmatrix}
      \bD^0_0 & \bD^0_1 & \dots  & \bD^0_{\LL-1} & \bD^0_\LL\\
      0       & \bD^1_1 & \dots  & \bD^1_{\LL-1}   & \bD^1_\LL\\
      \vdots & \vdots & \ddots & \vdots & \vdots \\
      0 & 0   & \dots  & \bD^{\LL-1}_{\LL-1} & \bD^{\LL-1}_\LL\\
      0 & 0   &\dots  & 0 &\bD_\LL^\LL\\
    \end{bmatrix}.
\end{equation}
Note that, in formulas~\eqref{MatrixD},\;\eqref{ExtendedMatrixD},
the symbol `$0$' must be considered as a zero block 
of the
corresponding size.


\subsection{Some properties of the matrix}\label{SubsectionSomeProperties}  

\subsubsection{Single function
       }\label{SubsectionPropertiesOfDL}

\begin{definition}\label{DefV}
Let $\LL\in\Z_{\ge0}$, a function $f:\C^d\to\C^d$ be sufficiently differentiable, and $x_0$ be a point of $\C^d$.
By $V_\LL$ denote the {\em null-space (kernel)} of the matrix $\cc\bD_\LL f(x_0)$:
\begin{equation}\label{VDefinition}
  V_\LL:=\ker\cc\bD_\LL f(x_0).
\end{equation}
\end{definition}

\begin{definition}
Let the null-space $V_\LL$ be given by~\eqref{VDefinition} and the set $\cc\cP_\LL$ be given
by~\eqref{ConcatenatedP(x)}; then by $\cV_\LL$ we denote the following polynomial space:
\begin{equation}\label{cVDefinition}
  \cV_\LL:=\set{\left[\cc\cP_\LL\right]v}{v\in V_\LL}.
\end{equation}
\end{definition}

\begin{remark}
The component-wise form of the matrix $\cc\bD_\LL$, $\LL\in\Zp$, is 
\begin{equation}\label{Component-WiseP}
  \left[\cc\bD_\LL\right]_{{\jj}{\kk},\ 1\le {\jj},{\kk}\le\cc d(\LL)}
  =
  \left\{%
  \begin{aligned}
    &(-i)^{\left|\ind{{\kk}}{\alpha}-\ind{{\jj}}{\beta}\right|}\dbinom{\ind{{\kk}}{\alpha}}{\ind{{\jj}}{\beta}}D^{\ind{{\kk}}{\alpha}-\ind{{\jj}}{\beta}}, &&
                                                                    \ind{{\jj}}{\beta}\le\ind{{\kk}}{\alpha}\,,\\
    &\quad0,                                        && \mbox{otherwise},
  \end{aligned}%
  \right.\qquad
\end{equation}
where $\left(\ind{1}{\alpha},\dots,\ind{\cc d(\LL)}{\alpha},\right)
    = \left(\ind{1}{\beta},\dots,\ind{\cc d(\LL)}{\beta},\right)=\cc\cA_\LL$.
%
\end{remark}

\begin{theorem}\label{TheoremDfg}
Let $\LL\in\Zp$.
Let the set $\cc\cD_\LL$ be given
by~\eqref{ConcatenatedD(x)}, the matrix $\cc{\bD}_\LL$ be given by~\eqref{ExtendedMatrixD},
and functions $f,g$ be sufficiently differentiable. 
Then we have 
\begin{equation}\label{Dfg}
  \left[\cc\cD_\LL(fg)\right] = \left[\cc\cD_\LL f\right]\cc\bD_\LL g = \left[\cc\cD_\LL g\right]\cc\bD_\LL f,
\end{equation}
\end{theorem}

Here we omit the proof of formula~\eqref{Dfg}
and note only that
the formula
is a direct consequence of formula~\eqref{Component-WiseP} and
the Leibniz rule, see~\eqref{MultidimensionalLeibnizRule}. 

Now we investigate ranks of submatrices in the upper right corner of the matrix
$\cc\bD_\LL f(x_0)$ (the matrices $\bD^{{\lL}'}_{\lL} f(x_0)$,
${\lL}'=0,\dots,{\lL}$, ${\lL}=0,\dots,\LL$, given by~\eqref{MatrixDlL});
where the function $f:\C^d\to\C^d$ is sufficiently differentiable and $x_0\in\C^d$ is a given
point.

\begin{proposition}\label{Remark_OnDiagonals}
Let $\LL\in\Zp$. The submatrix $\bD^{{\lL}'}_{{\lL}}$, ${\lL}'=0,\dots,{\lL}$, ${\lL}=0,\dots,\LL$,
contains only the derivatives of order ${\lL}-{\lL}'$.
\end{proposition}

\begin{corollary}\label{Corollary_Diagonals}
All the submatrices on the ${\lL}$th block diagonal of the matrix
$\cc\bD_\LL$, i.\,e., the submatrices $\bD^{0}_{{\lL}},\bD^{1}_{{\lL}+1},\dots,\bD^{\LL-{\lL}}_{\LL}$,
contain the derivatives of order ${\lL}$.
\end{corollary}

It easy to see that the ${\jj}$th, $1\le {\jj}\le d({\lL})$, row of the
matrix $\bD_{\lL}^{\lL} f(x_0)$, ${\lL}=0,\dots,\LL$, contains only one non-zero element $f(x_0)$ (if $f(x_0)\ne0$),
which is situated on the ${\jj}$th position. Thus $\bD_{\lL}^{\lL} f(x_0)=\bI f(x_0)$,
where $\bI$ is the $d({\lL})\times d({\lL})$ identity
matrix. Since the matrix $\cc\bD_\LL f(x_0)$ is an upper triangular matrix,
we can state a necessary and sufficient condition that the matrix $\cc\bD_\LL f(x_0)$ is singular.
\begin{theorem}\label{Statement:SingularConditiontoD}
The matrix $\cc\bD_\LL f(x_0)$, $\LL\in\Zp$, is singular iff $f(x_0)=0$.
\end{theorem}

Now state a theorem about ranks of all other blocks of the matrix $\cc\bD_\LL f(x_0)$.

\begin{theorem}\label{Theorem:RankOfDlL}
The $d({\lL}')\times d({\lL})$ submatrix $\bD^{{\lL}'}_{\lL} f(x_0)$,
${\lL}'=0,\dots,{\lL}-1$, ${\lL}=1,\dots,\LL$, $\LL\in\Zp$, has full rank, i.\,e., the rank of $\bD^{{\lL}'}_{\lL} 
f(x_0)$ is equal to $d({\lL}')$, if and only if there exists
at least one non-zero derivative $D^{\gamma}f(x_0)$,
$|\gamma|={\lL}-{\lL}'$.
\end{theorem}

The proof of Theorem~\ref{Theorem:RankOfDlL} is given in Appendix~\ref{AppendixTheorem:RankOfDlL}.

Hence we see that each of the submatrices $\bD^{{\lL}'}_{\lL} f(x_0)$,
${\lL}'=0,\dots,{\lL}$, ${\lL}=0,\dots,\LL$, is either a full rank matrix or zero matrix.

Finally the following theorem allows to determine the dimension of the null-space of $\cc\bD_\LL f(x_0)$.
\begin{theorem}\label{Lemma_DimensionNullSpace}
Let $\LL\in\Zp$.
Suppose
%
there exists a multi-index $\alpha\in\Zp^d$, $|\alpha|>0$,
such that $D^{\alpha} f(x_0)\ne0$
and, for any multi-index $\alpha'\in\Zp^d$ such
that $0\le|\alpha'|<|\alpha|$,
the derivative $D^{\alpha'} f(x_0)$ vanishes.
Then
\begin{equation}\label{NullSpaceDimension}
  \dim\ker\cc\bD_\LL f(x_0)
  = \left\{
    \begin{aligned}
        &\cc d(\LL)-\cc d(\LL-|\alpha|)&&\mbox{ if } \LL\ge|\alpha|>0;\\
        &\cc d(\LL) &&\mbox{ if } \LL<|\alpha|.
    \end{aligned}\right.
\end{equation}
\end{theorem}
The theorem is a direct consequence of the following lemma.
\begin{lemma}\label{Lemma_Rank}
Under the conditions of Theorem~\ref{Lemma_DimensionNullSpace}, we have
$$
 \rank\cc\bD_\LL f(x_0)
 = \left\{
    \begin{aligned}
        &\cc d(\LL-|\alpha|)&&\mbox{ if } \LL\ge|\alpha|>0;\\
        &0 &&\mbox{ if } \LL<|\alpha|.
    \end{aligned}\right.
$$
\end{lemma}
\begin{proof}[Sketch of the proof of Lemma~\protect\ref{Lemma_Rank}.]
For the case $\LL\ge|\alpha|>0$,
by Theorem~\ref{Theorem:RankOfDlL} and Corollary~\ref{Corollary_Diagonals},
each block on the $|\alpha|$th block diagonal of the matrix $\cc\bD_\LL f(x_0)$
is a full rank matrix. However, since the blocks of $\cc\bD_\LL f(x_0)$ are not, generally, square matrices;
the problem to determine the rank of the matrix is not trivial.

Using an analog of the Gaussian elimination algorithm (applied to columns instead of rows) and moving
(actually permutating)
zero columns to the left, the matrix $\cc\bD_\LL f(x_0)$ can always be transformed
into a {\em strictly} upper triangular matrix, where the lowest
non-zero diagonal goes to
the lower right corner of the last submatrix $\bD^{\LL-|\alpha|}_\LL$. So
$\rank\cc\bD_\LL f(x_0)=\cc d(\LL-|\alpha|)$.

Since, in the case $\LL<|\alpha|$, the matrix $\cc\bD_\LL f(x_0)$ is a zero matrix;
the case $\LL<|\alpha|$ is trivial.
\end{proof}



Introduce some notation.
We always can consider $\C^{\cc d(\LL)}$, $\LL\in\Zp$, as a space with a Cartesian
coordinate system,
where the dot product of the Cartesian coordinates
is $x \dotpr y :=\sum_{{\jj}}x_{\jj}\ov{y_{\jj}}$.
Namely, we have
$$
  \C^{\cc d(\LL)}:=\Span{e_{\jj}}{1\le {\jj}\le\cc d(\LL)},
$$
%
where 
the span
is over $\C$ and $e_{\jj}$ is the ${\jj}$th basis vector: 
$e_{\jj}:=\left(\delta_{{\jj}1},\dots,\delta_{{\jj},\cc d(\LL)}\right)$;
and we can decompose $\C^{\cc d(\LL)}$ as follows
\begin{equation}\label{EuclideanSpaceDecomposition}
  \C^{\cc d(\LL)}=\ssc{0}{\C^{\cc d(\LL)}}\oplus\ssc{1}{\C^{\cc d(\LL)}}\oplus\cdots\oplus\ssc{\LL}{\C^{\cc d(\LL)}},
\end{equation}
where
\begin{align*}
  &\ssc{0}{\C^{\cc d(\LL)}}:=\mathop{\mathrm{span}}\left\{e_1\right\},\\
  &\ssc{\lL}{\C^{\cc d(\LL)}}:=\Span{e_{\jj}}{\cc d({\lL}-1)+1\le {\jj} \le \cc d({\lL})},\qquad {\lL}=1,\dots,\LL,
\end{align*}
and the direct sums in~\eqref{EuclideanSpaceDecomposition} are orthogonal.
%
%
%
%
%
%
%


\begin{remark}
Decomposition~\eqref{EuclideanSpaceDecomposition} 
corresponds to
the block structure of the matrix $\cc\bD_\LL$, $\LL\in\Zp$,
(as well as structures of $\cc{\cP}_\LL$ and $\cc{\cD}_\LL$).
\end{remark}

\begin{definition}
By $\Pj_{\lL}$, denote
the {\em orthogonal projection} on the subspace $\ssc{{\lL}}{\C^{\cc d(\LL)}}$, ${\lL}=0,\dots,\LL$;
and define subspaces of the null-space $V_\LL$, see Definition~\ref{DefV}, as
\begin{equation}\label{SubspacelVDefinition}
  \ssc{{\lL}}{V_\LL} := \Pj_{\lL} V_\LL,\qquad {\lL}=0,1,\dots,\LL. 
\end{equation}
\end{definition}

\begin{remark}
Note that generally
$V_\LL$ 
is not a sum, like~\eqref{EuclideanSpaceDecomposition}, of 
$\ssc{{\lL}}{V_\LL}$. 
\end{remark}

Now we can formulate the following theorem.
\begin{theorem}\label{Lemma_non-zeroLSubspace}
Let $\LL\in\N$. Let the matrix $\cc\bD_\LL f(x_0)$ be singular and $V_\LL:=\ker \cc\bD_\LL f(x_0)$.
Then 
the subspace $\ssc{\LL}{V_\LL}:=\Pj_\LL V_\LL$ is non-zero.
\end{theorem}
\begin{proof}
If the matrix $\cc\bD_\LL f(x_0)$ is a zero matrix, there is nothing to prove.

Assume the converse, i.\,e., suppose $\ssc{\LL}{V_\LL}$ is a zero space; 
then
\begin{equation*}
  \dim\ker\cc\bD_\LL f(x_0) 
     = \dim\ker\cc\bD_{\LL-1}  f(x_0).
\end{equation*}

Since the matrix $\cc\bD_\LL f(x_0)$ is singular and non-zero; 
there exists a number ${\kk}\in\N$, $1\le {\kk}\le \LL$, 
that the block ${\kk}$-diagonal is the lowest non-zero block diagonal. 
By Theorem~\ref{Lemma_DimensionNullSpace}, we have
\begin{equation*}
  \dim\ker\cc\bD_{\LL-1}  f(x_0) = \cc d(\LL-1) - \cc d(\LL-1-{\kk})<\cc d(\LL) - \cc d(\LL-{\kk})=\dim\ker\cc\bD_{\LL}  f(x_0).
\end{equation*}
This contradiction proves the theorem. 
\end{proof}

\subsubsection{Several functions}

For several functions $f_1, f_2,\dots,f_{\NN}$, we must consider the following concatenated matrix
\begin{equation}\label{SeveralMatrices}
  \begin{bmatrix}
    \cc\bD_\LL f_1(x_0) \\ \vdots \\ \cc\bD_\LL f_{\NN}(x_0)
  \end{bmatrix}.
\end{equation}
For concatenated  matrix~\eqref{SeveralMatrices}, an analog of Theorem~\ref{Lemma_non-zeroLSubspace} is not valid.
We can state only the following proposition.

First we must restructure matrix~\eqref{SeveralMatrices} and
introduce a notation.

 Let ${\K},\LL\in\Zp$, ${\K}\le \LL$.
Define the ${\NN}\cc d(\LL)\times d({\K})$ matrix ${\te\bD}_{\K}$, where ${\NN}$ is the number of matrices
(functions),  as follows
\begin{equation}\label{MutlimatrixMatrixD}
  {\te\bD}_{\K}
  :=\begin{bmatrix}
      {\te\bD}^0_{\K} \\[0.5ex] {\te\bD}^1_{\K} \\ \vdots \\ {\te\bD}_{\K}^{\K} \\ 0 \\ \vdots \\ 0 \\
    \end{bmatrix},\quad\mbox{where }\
  {\te\bD}_{\K}^{\kk}
  := \begin{bmatrix}
      \bD_{\K}^{\kk}f_1(x_0) \\[0.5ex] \bD_{\K}^{\kk}f_2(x_0) \\ \vdots \\ \bD_{\K}^{\kk}f_{\NN}(x_0)
    \end{bmatrix},\quad {\kk}=0,1,\dots {\K}.
\end{equation}

Secondly formulate an analog of matrix~\eqref{ExtendedMatrixD} 
\begin{equation}\label{RestructuredConcatenatedMatrixDL}
  \cc{{\te\bD}}_\LL:=\begin{bmatrix}
      {\te\bD}_0 & {\te\bD}_1 & \dots & {\te\bD}_\LL\\
     \end{bmatrix}
      =\begin{bmatrix}
      {\te\bD}^0_0 & {\te\bD}^0_1 & \dots  & {\te\bD}^0_{\LL-1} & {\te\bD}^0_{\LL} \\
      0  & {\te\bD}^1_{1}        & \dots  & {\te\bD}^1_{\LL-1}   & {\te\bD}^1_\LL\\
      \vdots & \vdots & \ddots & \vdots & \vdots \\
      0 & 0 & \dots  & {\te\bD}^{\LL-1}_{\LL-1} & {\te\bD}^{\LL-1}_\LL\\
      0 & 0 & \dots  & 0 &{\te\bD}_\LL^\LL\\
    \end{bmatrix}.
\end{equation}




Finally we present a proposition.

\begin{proposition}\label{LemmaAboutDeltaMatricesInclusion}
Let $\LL\in\N$. 
Let the matrices $\cc{{\te\bD}}_\LL$, $\cc{{\te\bD}}_{\LL-1}$ be defined
by~\eqref{RestructuredConcatenatedMatrixDL}.
Let $V_\LL:=\ker\cc{{\te\bD}}_\LL$, $V_{\LL-1}:=\ker\cc{{\te\bD}}_{\LL-1}$ and subspaces $\ssc{\LL}{V_\LL}$,
$\ssc{\LL-1}{V_{\LL-1}}$ be given by~\eqref{SubspacelVDefinition}.
Suppose the subspace $\ssc{\LL}{V_\LL}$ is non-zero, then the subspace $\ssc{\LL-1}{V_{\LL-1}}$ is also non-zero.
\end{proposition}

The proof is given in Appendix~\ref{AppendixPropositionAboutDeltaMatricesInclusion}.

\begin{remark}
Note that Proposition~\ref{LemmaAboutDeltaMatricesInclusion} is also valid for
single matrix~\eqref{ExtendedMatrixD}.
\end{remark}



\section{Solution methods}\label{SectionSolutionMethods}

\subsection{Homogeneous PDE's}\label{SectionHomogeneousPDEs}

In this subsection, we consider 
{\em homogeneous} 
PDE's
with constant coefficients from the field $\C$. Note the polynomials  that induce the differential equations are not
necessary homogeneous; i.\,e., the polynomials are, in general, sums of terms of different degree.


\subsubsection{One equation}


\begin{theorem}\label{Theorem:PolynomialFromKernel}
Let $P$ 
be a polynomial with constant coefficients from the field $\C$.
Let $x_0$ be a point of $\C^d$. 
Let the matrix $\cc\bD_\LL$, $\LL\in\Zp$, be given
by~\eqref{ExtendedMatrixD}. Then the following non-zero function $e^{ix_0\dotpr x}\left[\cc\cP_\LL(x)\right] v$,
$v\in\C^{\cc d(\LL)}$, belongs to
$\ker P(-iD)$ if and only if $P(x_0)=0$ and
$v\in\ker\cc\bD_\LL P(x_0)$. 
\end{theorem}
%

\begin{proof}
As it has been said, see Theorem~\ref{Statement:SingularConditiontoD}, a necessary and
sufficient condition for the matrix $\cc\bD_\LL P(x_0)$ to be
singular is that the point $x_0\in\C^d$ be a root of the polynomial
$P$.

Consider the function
$$
  f(x):=P(-iD)\left(e^{ix_0\dotpr x}\left[\cc\cP_\LL(x)\right]v\right). 
$$
Taking the Fourier transform of the previous function,
we obtain 
\begin{equation}\label{DistributionFT}
  \hat f(\xi):=P(\xi)\left[\cc\cD_\LL\delta(\xi-x_0)\right]v,\qquad \xi\in\C^d.
\end{equation}
The adjoint operator (set of operators) $\cc\cD_\LL^{\,\ast}$ satisfies a property
\begin{equation}\label{AdjointOperator}
  \cc\cD_\LL^{\,\ast} = \ov{\cc\cD_\LL}
\end{equation}
(
the adjunction, like the complex conjugation, 
is distributive over the comma).
Using Definition~\ref{ComplexValuedDistributionDefinition}, Theorem~\ref{TheoremDfg}, and property~\eqref{AdjointOperator};
for any test function $\phi\in S(\C^d)$, 
the functional $T_{\hat f}(\phi)$ (where $\hat f$ is distribution~\eqref{DistributionFT}) is of the form
\begin{align}
  &T_{\hat f}(\phi)=\ov{\ip{P(\cdot)\left[\cc\cD_\LL\delta(\cdot-x_0)\right]v}{\phi}}
  =\ov{\ip{\delta(\cdot-x_0)}{\left[\ov{\cc\cD_\LL}\left(\ov{P}\phi\right)\right]\ov{v}}}\nonumber\\
  &\qquad=\left[\cc\cD_\LL\left(P(x_0)\ov{\phi(x_0)}\right)\right]v
             \ \stackrel{\mbox{by~\eqref{Dfg}}}{=}\
                 \left[\cc\cD_\LL\ov{\phi(x_0)}\right]\left[\cc\bD_\LL P(x_0)\right]v.
    \label{qqq}
\end{align}

Using the expression in the right-hand side of \eqref{qqq}, the proof 
is trivial.
\end{proof}

\begin{remark}
Theorem~\ref{Theorem:PolynomialFromKernel} can be 
proved for other (linear and linear-conjugate,
see for example~\cite{KolmogorovFomin})
functionals like~\eqref{ComplexValuedDistribution}.
\end{remark}






\begin{theorem}\label{PolynomialSpasesInclusion}
Let $\LL\in\Zp$. Let $P$ be a polynomial. Let $x_0\in\C^d$ be a root of the polynomial $P$.
Let the matrices $\cc\bD_\LL$, $\cc\bD_{\LL+1}$ be given by~\eqref{ExtendedMatrixD}.
Let the polynomial spaces be of the form
\begin{align*}
  &\cV_\LL:=\set{\left[\cc\cP_\LL\right]v}{v\in \ker\cc\bD_{\LL}P(x_0)},\\
  &\cV_{\LL+1}:=\set{\left[\cc\cP_{\LL+1}\right]v}{v\in\ker\cc\bD_{\LL+1}P(x_0)}.
\end{align*}
Then
\begin{equation}\label{PolynomialSpacesInclusion}
  \cV_\LL=\Pi_{\le \LL}\cap\cV_{\LL+1}.
\end{equation}
\end{theorem}

The proof of Theorem~\ref{PolynomialSpasesInclusion} is given in Appendix~\ref{AppendixPolynomialSpasesInclusion}.

\begin{remark}
%
The previous theorem reflects a 
property of differentiation
to commutate with translation. 
\end{remark}

Below we state a corollary of
Theorem~\ref{Lemma_non-zeroLSubspace} and Theorem~\ref{PolynomialSpasesInclusion}.
\begin{corollary}\label{Corollary_InfiniteDegreePolynomials}
Under the conditions of Theorem~\ref{PolynomialSpasesInclusion},
we see that
the null-space of the operator $P(-iD)$ contains polynomials (multiplied by the exponential
$e^{ix_0\dotpr x}$) up to an arbitrary large total degree.
\end{corollary}

\begin{remark}
Corollary~\ref{Corollary_InfiniteDegreePolynomials} expresses a fundamental property, see~\cite{AbramovPetkovsek},
of polynomial solutions to a single PDE with constant coefficients.
\end{remark}


\begin{theorem}\label{Theorem2cases}
Let $\LL\in\Zp$. Let $P$ be an algebraic polynomial, let $x_0$ be a root of $P$.
Let $D^\alpha P(x_0)$, $\alpha\in\Z^d_{\ge0}$, be a non-zero derivative of the least order.
Then we have
\begin{equation*}
  \dim\left(e^{ix_0\dotpr x}\Pi_{\le \LL}\cap\ker P(-iD)\right)
  = \left\{
    \begin{aligned}
        &\cc d(\LL)-\cc d(\LL-|\alpha|)&&\mbox{ if } \LL\ge|\alpha|>0;\\
        &\cc d(\LL) &&\mbox{ if } \LL<|\alpha|.
    \end{aligned}\right.
\end{equation*}
Moreover, if $\LL<|\alpha|$; then
$$
  e^{ix_0\dotpr x}\Pi_{\le \LL}\cap\ker P(-iD) =  e^{ix_0\dotpr x}\Pi_{\le \LL}.
$$
\end{theorem}
The previous theorem is a direct consequence of Theorem~\ref{Lemma_DimensionNullSpace}.

\subsubsection{System of equations}

For the case of a system of PDE's, i.\,e., if we have several algebraic polynomials $P_{n'}$, $n'=1,2,\dots,{\NN}$,
the polynomial solution of the corresponding system
\begin{equation}\label{PDE's}
  \left\{
    \begin{aligned}
      P_1&(-iD)\cdot=0,\\
      &\vdots\\
      P_{\NN}&(-iD)\cdot=0\\
    \end{aligned}
    \right.
\end{equation}
is the intersection of the null-spaces
$\ker P_{n'}(-iD)$, $n'=1,\dots,{\NN}$.

Namely we have the following proposition.
\begin{proposition}
Let $P_{n'}$, $n'=1,2,\dots,{\NN}$, be algebraic polynomials.
The non-zero expression $e^{ix_0\dotpr x}\left[\cc\cP_\LL\right]v$, $\LL\in\Zp$,
where $v\in\C^{\cc d(\LL)}$, belongs to the null-space of system~\eqref{PDE's}   
iff
the vector $v$ belongs to
the null-space of the block matrix
$$
  \begin{bmatrix}
    \cc\bD_\LL P_1(x_0)\\
    \vdots\\
    \cc\bD_\LL P_{\NN}(x_0)
  \end{bmatrix}
$$
and $P_1(x_0)=\cdots=P_{\NN}(x_0)=0$.
\end{proposition}

\subsection{PDE with polynomial right-hand side}

In the previous subsection, we have considered polynomial
solutions of homogeneous constant coefficient PDE's; i.\,e., actually, PDE's have zero right-hand
sides. Nevertheless the matrix approach allows generalizing PDE's
to polynomial right-hand sides (multiplied by an exponential).

\begin{theorem}\label{Theorem:PolynomialOnRight-HandSide}
Let $\LL\in\Zp$. Let $P,F$ be algebraic polynomials with constant coefficients from $\C$,
$x_0\in\C^d$ be a root of $P$, $\alpha\in\Zp^d$ be a multi-index that defines the least order derivative
such that $D^\alpha P(x_0)\ne0$. Let the polynomial $F$ be defined as follows: $F(x):=\left[\cc\cP_{\deg F}(x)\right]w$,
$w\in\C^{\cc d(\deg F)}$. Let $\LL\ge\deg F+|\alpha|$, the matrix $\cc\bD_\LL$ be given by~\eqref{ExtendedMatrixD}.
Let $v\in\C^{\cc d(\LL)}$ be a column vector and $p:=\left[\cc\cP_{\LL}\right]v$ be the corresponding polynomial. Then
the algebraic polynomial $p$   
is a solution to PDE
\begin{equation}\label{PDEF}
  P(-iD)\left(e^{ix_0\dotpr x}p(x)\right)=e^{ix_0\dotpr x}F(x)
\end{equation}
iff
the vector $v$  
is a solution to linear algebraic equation
\begin{equation}\label{LinearSystem}
  \left[\cc\bD_\LL P(x_0)\right]v
  =\begin{bmatrix}
    w^T &
    \underbrace{\begin{array}{ccc} 0 & \cdots & 0\end{array}}_{\cc d(\LL)-\cc d\left(\deg F\right)}
  \end{bmatrix}^T
%
%
\end{equation}
\end{theorem}
\begin{remark}
Under the conditions of the previous theorem, introduce notation.
Denote by $V_\LL$ a linear space such that any vector $v\in V_\LL$ is a solution to linear algebraic
equation~\eqref{LinearSystem} and by $\cV_\LL:=\set{\left[\cc\cP_\LL\right]v}{v\in V_\LL}$ denote
the corresponding space of polynomials that the polynomials are solutions to PDE~\eqref{PDEF}.
So we can state 3 remarks:
\begin{enumerate}
\item
$\dim\cV_\LL=\cc d(\LL)-\cc d(\LL-|\alpha|)$;
\item
for any polynomial $F$ such that $\deg F\le \LL-|\alpha|$, linear system~\eqref{LinearSystem} is consistent;
\item
there exists a polynomial of an arbitrary large degree to satisfy 
PDE~\eqref{PDEF}.
\end{enumerate}
\end{remark}

However, in the previous theorem, the point $x_0\in\C^d$ can be no root of $P$. Below we state the corresponding
remark.


\begin{remark}
Under the conditions of the previous theorem and suppose $P(x_0)\ne0$; we see that
the polynomial $p$ is defined uniquely and
does not depend on the choice of $\LL$.
\end{remark}


The proof of the previous theorem and remarks is left to the reader. Note only that the theorem and remarks
are based on Theorem~\ref{Lemma_DimensionNullSpace}, Theorem~\ref{Theorem:PolynomialFromKernel},
and the classical Rouch\'e-Capelli theorem.





\section{Examples}\label{SectionExamples}
\newcounter{Example}

Below, for a function of 2 variables $f=f(x,y)$, $x,y\in\C$,
we present a component-wise form of the matrix $\cc\bD_3f(x_0,y_0)$, $(x_0,y_0)\in\C^2$:
{\arraycolsep=0.3em
\begin{align*}
&\cc\bD_3f(x_0,y_0)\\
&=\left.\left[
\begin{array}{c|cc|ccc|cccc}
 f & -i f_x & -i f_y & -f_{xx} & -f_{xy} &
   -f_{yy} & i f_{xxx} & i f_{xxy} & i f_{xyy} & i
    f_{yyy} \\
 \hline
 0 & f & 0 & -2 i f_x & -i f_y & 0 & -3 f_{xx} & -2
   f_{xy} & -f_{yy} & 0 \\
 0 & 0 & f & 0 & -i f_x & -2 i f_y & 0 & -f_{xx} &
   -2 f_{xy} & -3 f_{yy} \\
 \cline{2-10}
 \multicolumn{1}{c}{0} & 0 & 0 & f & 0 & 0 & -3 i f_x & -i f_y & 0 & 0 \\
 \multicolumn{1}{c}{0} & 0 & 0 & 0 & f & 0 & 0 & -2 i f_x & -2 i f_y & 0 \\
 \multicolumn{1}{c}{0} & 0 & 0 & 0 & 0 & f & 0 & 0 & -i f_x & -3 i f_y \\
 \cline{4-10}
 \multicolumn{1}{c}{0} & 0 & \multicolumn{1}{c}{0} & 0 & 0 & 0 & f & 0 & 0 & 0 \\
 \multicolumn{1}{c}{0} & 0 & \multicolumn{1}{c}{0} & 0 & 0 & 0 & 0 & f & 0 & 0 \\
 \multicolumn{1}{c}{0} & 0 & \multicolumn{1}{c}{0} & 0 & 0 & 0 & 0 & 0 & f & 0 \\
 \multicolumn{1}{c}{0} & 0 & \multicolumn{1}{c}{0} & 0 & 0 & 0 & 0 & 0 & 0 & f \\
\end{array}
\right]\right|_{\displaystyle\begin{aligned} &(x,y)\\[-0.01\textheight]&=(x_0,y_0)\end{aligned}}\!\!.
\end{align*}
}

\subsection{Homogeneous equations}

First we consider 3 examples of polynomial solution to PDE's, presented in the paper~\cite{Pedersen}.

\paragraph{\addtocounter{Example}{1}Example~\arabic{Example}}

The first example is 2D Laplace operator
\begin{equation}\label{2DLaplaceOperator}
  L_1:=\frac{\partial^2}{\partial x^2}+\frac{\partial^2}{\partial y^2},\qquad x,y\in\R.
\end{equation}
So the polynomial that induces the operator
is $P_1(x,y):=-x^2-y^2$, $x,y\in\C$, Thus the $10\times 10$ matrix $\cc\bD_3 P_1(0,0)$ is of the form
\begin{equation}\label{LaplaceOperatorMatrix}
\cc\bD_3 P_1(0,0) :=
\left[
\begin{array}{c|cc|ccc|cccc}
 0 & 0 & 0 & 2 & 0 & 2 & 0 & 0 & 0 & 0 \\
 \hline
 0 & 0 & 0 & 0 & 0 & 0 & 6 & 0 & 2 & 0 \\
 0 & 0 & 0 & 0 & 0 & 0 & 0 & 2 & 0 & 6 \\
 \cline{2-10}
 \multicolumn{1}{c}{0}& 0 & 0 & 0 & 0 & 0 & 0 & 0 & 0 & 0 \\
 \multicolumn{1}{c}{0} & 0 & 0 & 0 & 0 & 0 & 0 & 0 & 0 & 0 \\
 \multicolumn{1}{c}{0} & 0 & 0 & 0 & 0 & 0 & 0 & 0 & 0 & 0 \\
 \cline{4-10}
 \multicolumn{1}{c}{0} & 0 & \multicolumn{1}{c}{0} & 0 & 0 & 0 & 0 & 0 & 0 & 0 \\
 \multicolumn{1}{c}{0} & 0 & \multicolumn{1}{c}{0} & 0 & 0 & 0 & 0 & 0 & 0 & 0 \\
 \multicolumn{1}{c}{0} & 0 & \multicolumn{1}{c}{0} & 0 & 0 & 0 & 0 & 0 & 0 & 0 \\
 \multicolumn{1}{c}{0} & 0 & \multicolumn{1}{c}{0} & 0 & 0 & 0 & 0 & 0 & 0 & 0
\end{array}
\right].
\end{equation}
Now multiplying row vector $\left[\cc\cP_3\right]$, see~\eqref{ConcatenatedP(x)}, on the right
by a $7$-column matrix (null-space) $\left[\ker\cc\bD_3 P_1(0,0)\right]$;
we obtain the well-known result
\begin{align}
  &\left[\begin{array}{c|cc|ccc|cccc}
    1 & x & y & x^2 & xy & y^2 & x^3 & x^2y & xy^2 & y^3
         \end{array}\right]
\left[
\begin{array}{ccccccc}
 1 & 0 & 0 & 0 & 0 & 0 & 0 \\
 \hline
 0 & 1 & 0 & 0 & 0 & 0 & 0 \\
 0 & 0 & 1 & 0 & 0 & 0 & 0 \\
 \hline
 0 & 0 & 0 & 0 & -1 & 0 & 0 \\
 0 & 0 & 0 & 1 & 0 & 0 & 0 \\
 0 & 0 & 0 & 0 & 1 & 0 & 0 \\
 \hline
 0 & 0 & 0 & 0 & 0 & -1 & 0 \\
 0 & 0 & 0 & 0 & 0 & 0 & -3 \\
 0 & 0 & 0 & 0 & 0 & 3 & 0 \\
 0 & 0 & 0 & 0 & 0 & 0 & 1
\end{array}
\right] \nonumber \\
  &\qquad\qquad\qquad\qquad =
  \begin{bmatrix}
    1&x&y&x y&y^2-x^2&3 x y^2-x^3&y^3-3 x^2 y
  \end{bmatrix}.
  \label{WellKnownPolynomials}
\end{align}

\paragraph{\addtocounter{Example}{1}Example~\arabic{Example}}

The polynomial $P_2(x,y):=-x^2-iy$ that induces the operator  of this example
$L_2:=\frac{\partial^2}{\partial x^2}-\frac{\partial}{\partial y}$ is not homogeneous.
The matrix $\cc\bD_3 P_2(0,0)$ is of the form
$$
\cc\bD_3 P_2(0,0):=
\left[
\begin{array}{c|cc|ccc|cccc}
 0 & 0 & -1 & 2 & 0 & 0 & 0 & 0 & 0 & 0 \\
 \hline
 0 & 0 & 0 & 0 & -1 & 0 & 6 & 0 & 0 & 0 \\
 0 & 0 & 0 & 0 & 0 & -2 & 0 & 2 & 0 & 0 \\
 \cline{2-10}
 \multicolumn{1}{c}{0} & 0 & 0 & 0 & 0 & 0 & 0 & -1 & 0 & 0 \\
 \multicolumn{1}{c}{0} & 0 & 0 & 0 & 0 & 0 & 0 & 0 & -2 & 0 \\
 \multicolumn{1}{c}{0} & 0 & 0 & 0 & 0 & 0 & 0 & 0 & 0 & -3 \\
 \cline{4-10}
 \multicolumn{1}{c}{0} & 0 & \multicolumn{1}{c}{0} & 0 & 0 & 0 & 0 & 0 & 0 & 0 \\
 \multicolumn{1}{c}{0} & 0 & \multicolumn{1}{c}{0} & 0 & 0 & 0 & 0 & 0 & 0 & 0 \\
 \multicolumn{1}{c}{0} & 0 & \multicolumn{1}{c}{0} & 0 & 0 & 0 & 0 & 0 & 0 & 0 \\
 \multicolumn{1}{c}{0} & 0 & \multicolumn{1}{c}{0} & 0 & 0 & 0 & 0 & 0 & 0 & 0
\end{array}
\right]
$$
and a basis of the corresponding polynomial space $P_2(-iD)\cap\Pi_{\le 3}$ is
$$
  \left\{1,x,x^2+2 y,x^3+6 x y\right\}.
$$

\paragraph{\addtocounter{Example}{1}Example~\arabic{Example}} The third example taken 
from the paper~\cite{Pedersen} is interesting from two points
of view. Namely we have a system of two operators: an elliptic (the Laplace operator)
$L_3:=\frac{\partial^2}{\partial x^2}+\frac{\partial^2}{\partial y^2}+\frac{\partial^2}{\partial z^2}$
and hyperbolic $L_4:=\frac{\partial}{\partial x}\frac{\partial}{\partial y}+
\frac{\partial}{\partial x}\frac{\partial}{\partial z}+\frac{\partial}{\partial y}\frac{\partial}{\partial z}$. And
it is a 3D example.

Since, in~\cite{Pedersen}, the 3rd degree polynomials are considered only;
therefore, we use the last (block) columns $\bD_3P_i(0,0)
  :=\begin{bmatrix}
      \bD^0_3P_i(0,0)\\ \vdots\\ \bD_3^3P_i(0,0) 
    \end{bmatrix},\ i=3,4$, (see~\eqref{MatrixD},~\eqref{MatrixDlL}) of the matrices $\cc\bD_3P_3(0,0)$,
$\cc\bD_3P_4(0,0)$,
where $P_3,P_4$ are the algebraic polynomials that corresponds to the operators $L_3,L_4$, respectively.
(It is possible because of
affinely invariance of the polynomial solution.) Below we present the following $6\times10$ matrix (because other
blocks are zero matrices):
\begin{equation}\label{3DSystem}
  \left[\begin{array}{c}
    \bD^1_3 P_3(0,0) \\
    \hline
    \bD_3^1 P_4(0,0) \\
  \end{array}\right] :=
  \left[
\begin{array}{cccccccccc}
 6 & 0 & 2 & 0 & 0 & 0 & 0 & 2 & 0 & 0 \\
 0 & 2 & 0 & 6 & 0 & 0 & 0 & 0 & 2 & 0 \\
 0 & 0 & 0 & 0 & 2 & 0 & 2 & 0 & 0 & 6 \\
 \hline
 0 & 2 & 0 & 0 & 2 & 1 & 0 & 0 & 0 & 0 \\
 0 & 0 & 2 & 0 & 0 & 1 & 2 & 0 & 0 & 0 \\
 0 & 0 & 0 & 0 & 0 & 1 & 0 & 2 & 2 & 0 \\
\end{array}
\right]
\end{equation}
Using the null-space of matrix~\eqref{3DSystem}, we get a basis of the space $P_3(-iD)\cap P_4(-iD)\cap\Pi_3$:
\begin{align*}
&\left\{3 x^2 y-3 x^2 z-y^3+z^3,-x^3+3 x^2 y+3 x y^2-6 x y z-2
   y^3+3 y z^2\right.,\\
 &\qquad\qquad-2 x^3+3 x^2 y+3 x y^2-6 x y z+3 x z^2-y^3, \\
   &\hspace{0.25\textwidth}\left.x^3+3
   x^2 y-3 x^2 z-3 x y^2-y^3+3 y^2 z\right\}.
\end{align*}
Note that, in the paper~\cite{Pedersen}, another basis is presented. But it is not hard
to see that our own and Pedersen's, see~\cite{Pedersen}, basis are bases of the same space.

Secondly we present two examples, where other (not the origin) roots  of polynomials are used.

\paragraph{\addtocounter{Example}{1}Example~\arabic{Example}}

This example is taken from the paper~\cite{Zakh20}.
Since the null-space of the symbol of 2D Laplace operator~\eqref{2DLaplaceOperator}
is a 2D, if we shall consider the null-space in $\C^2$, manifold;
we can take another root of $P_1$. Here we take, as an example, a root $(1,i)$;
and we obtain the following matrix
\begin{equation*}
\cc\bD_3 P_1(1,i):=
\left[
\begin{array}{c|cc|ccc|cccc}
 0 & -2 i & 2 & -2 & 0 & -2 & 0 & 0 & 0 & 0 \\
 \hline
 0 & 0 & 0 & -4 i & 2 & 0 & -6 & 0 & -2 & 0 \\
 0 & 0 & 0 & 0 & -2 i & 4 & 0 & -2 & 0 & -6 \\
 \cline{2-10}
 \multicolumn{1}{c}{0} & 0 & 0 & 0 & 0 & 0 & -6 i & 2 & 0 & 0 \\
 \multicolumn{1}{c}{0} & 0 & 0 & 0 & 0 & 0 & 0 & -4 i & 4 & 0 \\
 \multicolumn{1}{c}{0} & 0 & 0 & 0 & 0 & 0 & 0 & 0 & -2 i & 6 \\
 \cline{4-10}
 \multicolumn{1}{c}{0} & 0 & \multicolumn{1}{c}{0} & 0 & 0 & 0 & 0 & 0 & 0 & 0 \\
 \multicolumn{1}{c}{0} & 0 & \multicolumn{1}{c}{0} & 0 & 0 & 0 & 0 & 0 & 0 & 0 \\
 \multicolumn{1}{c}{0} & 0 & \multicolumn{1}{c}{0} & 0 & 0 & 0 & 0 & 0 & 0 & 0 \\
 \multicolumn{1}{c}{0} & 0 & \multicolumn{1}{c}{0} & 0 & 0 & 0 & 0 & 0 & 0 & 0
\end{array}
\right]
\end{equation*}
and
subspace of Laplace's operator null-space:
\begin{align}
  e^{ix-y}&\Pi_{\le3}\cap\ker L_1 \nonumber \\
  &=e^{ix-y}\mathop{\mathrm{span}}\left\{1,x+iy,x^2+2 i x y-y^2,x^3+3 i x^2 y-3 x y^2-i y^3\right\}.
  \label{LaplaceAnotherRoot}
\end{align}
Note the real or imaginary parts of the polynomials in~\eqref{LaplaceAnotherRoot} (cf.~\eqref{WellKnownPolynomials}
and~\eqref{LaplaceAnotherRoot}) multiplied by
an exponential, do not become solutions of the Laplace operator.

\paragraph{\addtocounter{Example}{1}Example~\arabic{Example}}

This example is also taken from the paper~\cite{Zakh20}. Namely we shall consider
Laplace operator~\eqref{2DLaplaceOperator} shifted by a vector $(1,i)$:
$$
  L_\arabic{Example}:=\left(\frac{\partial}{\partial x}-\Id\right)^2+\left(\frac{\partial}{\partial y}-i\Id\right)^2,
  \qquad\text{where $\Id$ is the identity operator}.
$$
However the corresponding polynomial $P_\arabic{Example}$ will be of the form
$$
  P_\arabic{Example}(x,y):=\left(ix-1\right)^2+\left(iy-i\right)^2
$$
and the matrix $\cc\bD_3P_\arabic{Example}(-i,1)$ coincides with matrix~\eqref{LaplaceOperatorMatrix}.
Thus we have
$$
  e^{x+iy}\Pi_{\le3}\cap L_\arabic{Example} = e^{x+iy}\left(\Pi_{\le3}\cap L_1\right),
$$
where $L_1$ is Laplace operator~\eqref{2DLaplaceOperator}.

\subsection{PDE with polynomial right-hand side}

Finally we discuss 
PDE's with polynomial (multiplied, in general,
by an exponential) right-hand sides.

\paragraph{\addtocounter{Example}{1}Example~\arabic{Example}}
Consider the Helmholtz operator
\begin{equation}\label{HelmholtzOperator}
  L_5:=\frac{\partial^2}{\partial x^2}+\frac{\partial^2}{\partial y^2}-\Id,
\end{equation}
where $\Id$ is the
identity operator; also consider PDE with a polynomial right-hand side
\begin{equation}\label{HelmholtzPDE}
  F(x,y):=2 + 3 x - 2 x y + y^2=\left[\cc\cP_2(x,y)\right]\begin{bmatrix} 2 & 3 & 0 & 0 & -2 & 1 \end{bmatrix}^T.
\end{equation}
Let $(x_0,y_0):=(0,0)$.
However the origin is not a root
of the corresponding polynomial
\begin{equation}\label{HelmholtzPolynomial}
  P_5(x,y):=-x^2-y^2-1;
\end{equation}
and, according to Theorem~\ref{Theorem:PolynomialOnRight-HandSide},
we can take $\LL=2$.
So the linear algebraic equation
is of the form
\begin{equation}\label{NonSingularMatrix}
\left[
\begin{array}{c|cc|ccc}
 -1 & 0 & 0 & 2 & 0 & 2 \\
 \hline
 0 & -1 & 0 & 0 & 0 & 0 \\
 0 & 0 & -1 & 0 & 0 & 0 \\
 \cline{2-6}
 \multicolumn{1}{c}{0} & 0 & 0 & -1 & 0 & 0 \\
 \multicolumn{1}{c}{0} & 0 & 0 & 0 & -1 & 0 \\
 \multicolumn{1}{c}{0} & 0 & 0 & 0 & 0 & -1
\end{array}
\right]v
= \left[
\begin{array}{c}
 2  \\
 \hline
 3  \\
 0  \\
 \hline
 0  \\
-2  \\
 1
\end{array}
\right],\qquad v\in\C^{\cc d(2)}.
\end{equation}
%
%
Since the matrix $\cc\bD_2P_5(0,0)$ in the left-hand side of~\eqref{NonSingularMatrix} is not singular,
it follows that
PDE $L_5\cdot=F$ 
has only the unique polynomial solution
$$
  2 x y-3 x-y^2-4.
$$

\paragraph{\addtocounter{Example}{1}Example~\arabic{Example}}

Here we consider Helmholtz operator~\eqref{HelmholtzOperator} and polynomial~\eqref{HelmholtzPDE} again, but
take another point $(x_0,y_0)=(i,0)$ that is a root of polynomial~\eqref{HelmholtzPolynomial}.
Since $\deg F=2$ and $|\alpha|=1$, we shall use $\LL=3$; and the dimension of polynomial space to solve PDE
\begin{equation}\label{PDEHelmholtz4D}
  L_5\left(e^{-x}\cdot\right)=e^{-x}\left(-2 x y+3 x+y^2+2\right)
\end{equation}
is $\cc d(3)-\cc d(2)=d(3)=4$.

Since the linear algebraic system $\cc\bD_3P_5(i,0)v=\begin{bmatrix} 2 & 3 & 0 & 0 & -2 & 1 \end{bmatrix}^T$,
$v\in\C^{\cc d(2)}$,
is undetermined and consistent; using the standard technics,
we obtain the solution of PDE~\eqref{PDEHelmholtz4D}
\begin{align*}
           v_4 \left(3 x y+y^3\right)+v_2 \left(x+y^2\right)+v_3y+v_1-\frac{x^2
   y}{2}+x^2&+\frac{x y^2}{2}-\frac{x y}{2}-2 y^2,\\
   &v_1,v_2,v_3,v_4\in\C.
\end{align*}

\paragraph{\addtocounter{Example}{1}Example~\arabic{Example}}

In this example, we consider Poisson's equation. 
However we use a point $(1,1)$ that is not root of $P_1$.
We solve the following PDE:
\begin{equation}\label{PoissonEquation}
  L_1\left(e^{i x+i y}\cdot\right)=e^{i x+i y}(3+x-2y),
\end{equation}
where $L_1$ is Laplace operator~\eqref{2DLaplaceOperator}.
The value $\LL=1$ will suffice and the linear system is very simple
$$
\left[
\begin{array}{c|cc}
 -2 & 2 i & 2 i \\
 \hline
 0 & -2 & 0 \\
 0 & 0 & -2
\end{array}
\right]v=
\left[
\begin{array}{c}
  3 \\
  \hline
  1 \\
  -2 \\
\end{array}
\right]
$$
and the unique polynomial solution to PDE~\eqref{PoissonEquation} 
is of the form
$$
  -\frac{x}{2}+y-\left(\frac{3}{2}-\frac{i}{2}\right).
$$


\appendix

\section{Proof of Theorem~\ref{Theorem:RankOfDlL}}\label{AppendixTheorem:RankOfDlL}


First consider the matrix $\bD^0_{\lL} f(x_0)$, ${\lL}=1,\dots,\LL$.
We see that $\bD^0_{\lL} f(x_0)$ is a row vector
$\left[(-i)^{\lL} D^{\ind{{\jj}}{\alpha}} f(x_0)\right]$ $(1\le {\jj}\le d({\lL}))$.
Consequently the matrix $\bD^0_{\lL} f(x_0)$ has non-zero rank iff
there exists at least one multi-index $\ind{{\jj}}{\alpha}\in\cA_{\lL}$
such that $D^{\ind{{\jj}}{\alpha}}f(x_0)\ne0$.

Secondly
consider the matrices $\bD^{{\lL}'}_{\lL} f(x_0)$, ${\lL}'=1,\dots,{\lL}-1$, ${\lL}=1,\dots,\LL$. Note that
$\bD^{{\lL}'}_{\lL} f(x_0)$ is a $d({\lL}')\times d({\lL})$ matrix.
By $\ind{{\jj}{\kk}}{\gamma}$, $1\le {\jj}\le d({\lL}')$, $1\le {\kk}\le d({\lL})$, denote the difference 
$\ind{{\jj}}{\alpha}-\ind{{\kk}}{\beta}$, where $\left(\ind{1}{\alpha},\dots,\ind{d({\lL})}{\alpha}\right)={\cA}_{\lL}$
and $\left(\ind{1}{\beta},\dots,\ind{d({\lL}')}{\beta}\right)={\cA}_{{\lL}'}$.
Define auxiliary $d({\lL}')\times d({\lL})$ matrices $\bG^{{\lL}'}_{\lL}$, ${\lL}'=1,\dots,{\lL}-1$, ${\lL}=1,\dots,\LL$, as
follows
\begin{equation}\label{MatrixGamma}
  \begin{bmatrix}
    \bG^{{\lL}'}_{\lL}
  \end{bmatrix}_{{\jj}{\kk},\ 1\le {\jj}\le d({\lL}'),\, 1\le {\kk}\le d({\lL})}
  =\ind{{\jj}{\kk}}{\gamma}.
\end{equation}
We suppose that some entries of the matrix $\bG^{{\lL}'}_{\lL}$ do not
belong to $\Z^d_{\ge0}$, i.\,e., a tuple $\ind{{\jj}{\kk}}\gamma$ can contain
negative components.

\begin{lemma}\label{LemmaAboutGammaMatrix}
%
Any multi-index $\gamma\in\cA_{{\lL}-{\lL}'}$ appears once in each row and
no more than once in each column of the matrix $\bG^{{\lL}'}_{\lL}$.
However not every column of $\bG^{{\lL}'}_{\lL}$ contains the 
multi-index $\gamma$. 
\end{lemma}
\begin{proof}[Proof of Lemma~\ref{LemmaAboutGammaMatrix}.]
For any multi-index $\gamma\in\cA_{{\lL}-{\lL}'}$ and any row ${\jj}$,
${\jj}=1,\dots,d({\lL}')$, of the matrix $\bG^{{\lL}'}_{\lL}$, we can always  define
a $d$-tuple $\alpha\in\Z^d$ as $\alpha:=\gamma+\ssc{{\jj}}\beta$,
where $\ssc {\jj}\beta\in\cA_{{\lL}'}$. Since $\ssc
{\jj}\beta,\gamma\in\Z^d_{\ge0}$ and $|\ssc {\jj}\beta|={\lL}'$,
$|\gamma|={\lL}-{\lL}'$; therefore, $\alpha\in\Z^d_{\ge0}$, $|\alpha|={\lL}$.
Consequently there exists a unique (column) number
${\kk}\in\{1,\dots,d({\lL})\}$ such that $\alpha=\ind{{\kk}}{\alpha}\in\cA_{\lL}$.
%
%

Fix a column number ${\kk}\in\{1,\dots,d({\lL})\}$. Consider some
$\gamma\in\cA_{{\lL}-{\lL}'}$ and define a $d$-tuple $\beta$ as
$\beta:=\ind{{\kk}}\alpha-\gamma$, where $\ind{{\kk}}\alpha\in\cA_{\lL}$. If
$\gamma\le\ind{{\kk}}\alpha$, then $\beta=\ind{{\jj}}{\beta}\in\cA_{\lL}$ for
a unique (row)  number ${\jj}\in\{1,\dots,d({\lL}')\}$, else the column ${\kk}$
does not contain the multi-index $\gamma$.
\end{proof}

\begin{lemma}\label{LemmaAboutEntriesGammaMatrix}
For ${\lL}'<{\lL}$, let $\left(\ind1\beta,\dots,\ind{d({\lL}')}\beta\right)=\cA_{{\lL}'}$,
$\left(\ind1\gamma,\dots,\ind{d({\lL}-{\lL}')}\gamma\right)=\cA_{{\lL}-{\lL}'}$.
Consider an entry $\ind{{\jj}{\kk}}{\gamma}$ of the matrix $\bG_{\lL}^{{\lL}'}$
and suppose that $\ind{{\jj}{\kk}}{\gamma}$ is equal to a multi-index
$\ind{{\mm}}{\gamma}\in\cA_{{\lL}-{\lL}'}$, ${\mm}\in\{1,\dots,d({\lL}-{\lL}')\}$; then any entry
of the ${\kk}$th column of $\bG^{{\lL}'}_{\lL}$ below than
$\ind{{\jj}{\kk}}{\gamma}$, i.\,e., the entry $\ind{{\jj}'{\kk}}{\gamma}$,
${\jj}'\in\left\{{\jj}+1,\dots,d({\lL}')\right\}$, belongs either to the set of
the multi-indices $\{\ind{1}{\gamma},\dots,\ind{{\mm}-1}{\gamma}\}$ or
does not belong to $\Z^d_{\ge0}$.
\end{lemma}
\begin{proof}[Proof of Lemma~\ref{LemmaAboutEntriesGammaMatrix}.]
Since the multi-indices $\ind{{\jj}}{\beta}$, ${\jj}=1,\dots,d({\lL}')$, are
lexicographically ordered, i.\,e.,
$\ind{1}{\beta}\lxl\ind{2}{\beta}\lxl\cdots\lxl\ind{d({\lL}')}{\beta}$;
using~\eqref{MatrixGamma}, we obtain the order
\begin{equation}\label{Ordering}
  \ind{1{\kk}}{\gamma}\lxg\ind{2{\kk}}{\gamma}\lxg\cdots\lxg\ind{d({\lL}'),{\kk}}{\gamma}.
\end{equation}
Consider any $\ind{{\jj}'{\kk}}{\gamma}$, ${\jj}'\in\{{\jj}+1,\dots,d({\lL}')\}$,
assume the converse:
$\ind{{\jj}'{\kk}}{\gamma}\notin\{\ind{1}{\gamma},\dots,\ind{{\mm}-1}{\gamma}\}$
and $\ind{{\jj}'{\kk}}{\gamma}\in\Z^d_{\ge0}$. Then there exists a unique
number ${\mm}'\in\{{\mm},\dots,d({\lL}-{\lL})\}$ such that
$\ind{{\mm}'}{\gamma}=\ind{{\jj}'{\kk}}{\gamma}$. Since ${\mm}'\ge {\mm}$, we have
$\ind{{\mm}'}{\gamma}=\ind{{\jj}'{\kk}}{\gamma}\lxge\ind{{\mm}}{\gamma}=\ind{{\jj}{\kk}}{\gamma}$.
By~\eqref{Ordering}, we get ${\jj}'\le {\jj}$. This contradiction
concludes the proof.
\end{proof}

Finally let us prove the theorem. 
\begin{proof}[Proof of Theorem~\ref{Theorem:RankOfDlL}.]
Obviously that at least one non-zero derivative $D^{\gamma} f(x_0)$,
$|\gamma|={\lL}-{\lL}'$, is necessary for full rank of the
matrix $\bD^{{\lL}'}_{\lL} f(x_0)$.

Sufficiency.
For some $\ind{{\mm}}{\gamma}\in\cA_{{\lL}-{\lL}'}$, ${\mm}\in\{1,\dots,d({\lL}-{\lL}')\}$,
suppose that the derivative $D^{\ind{{\mm}}{\gamma}}f(x_0)\ne0$.
By Lemma~\ref{LemmaAboutGammaMatrix}, from the matrix $\bD^{{\lL}'}_{\lL} 
f(x_0)$, we can take a $d({\lL}')\times d({\lL}')$ submatrix with
the non-zero main diagonal $\left(D^{\ind{{\mm}}{\gamma}}f(x_0),\dots,D^{\ind{{\mm}}{\gamma}}f(x_0)\right)$
and we denote this submatrix by $\bS_{\mm}$.

Suppose, for $\ind{1}{\gamma}\in\cA_{{\lL}-{\lL}'}$, $D^{\ind{1}{\gamma}}f(x_0)\ne0$; then, by
Lemma~\ref{LemmaAboutEntriesGammaMatrix} and
property~\eqref{BinomialCoefficientVanishes}, all the entries
below the main diagonal of the submatrix $\bS_1$ vanish. So the
matrix $\bS_1$ is not singular, consequently $\rank \bD^{{\lL}'}_{\lL} f(x_0)=d({\lL}')$.

Otherwise, $D^{\ind{1}{\gamma}}f(x_0)=0$.
Suppose there exists a multi-index $\ind{{\mm}}{\gamma}\in\cA_{{\lL}-{\lL}'}$,
${\mm}\in\{2,\dots,d({\lL}-{\lL}')\}$, such that $D^{\ind{{\mm}}{\gamma}}f(x_0)\ne0$
and all the derivatives $D^{\ind{{\mm}'}{\gamma}}f(x_0)$,
${\mm}'=1,\dots,{\mm}-1$, vanish. Since all the entries below the
main diagonal of the matrix $\bS_{\mm}$ vanish, it follows that $\det
\bS_{\mm}\ne0$.
%
This concludes the proof of the sufficiency.
\end{proof}

\section{Proof of Proposition~\protect\ref{LemmaAboutDeltaMatricesInclusion}}
             \label{AppendixPropositionAboutDeltaMatricesInclusion}

\begin{proof}
Suppose a vector
$v\in\C^{\cc d(\LL)}$ belongs to $\ker\cc{{\te\bD}}_\LL$,
then we have the following system of equalities
\begin{equation}\label{VectorInNullSpace}
  \sum_{{\mm}=1}^{\cc d(\LL)}\dbinom{\ind{{\mm}}{\alpha}}{\ind{{\jj}}{\beta}}
      D^{\ind{{\mm}}{\alpha}-\ind{{\jj}}{\beta}}f_{\kk}(x_0)v_{\mm}=0,\qquad
  \begin{aligned}
    &{\kk}=1,\dots,{\NN},\\[-0.005\textheight]
    &{\jj}=1,\dots,\cc d(\LL),
  \end{aligned}
\quad \ind{{\mm}}{\alpha},\ind{{\jj}}{\beta}\in\cc\cA_\LL. 
\end{equation}
Consider some $\tau\in\Zp^d$, $|\tau|=1$, with a non-zero
component $\tau_s=1$, $s\in\{1,\dots,d\}$;
then, for multi-indices $\ind{{\jj}}{\beta}$, $\ind{{\mm}}{\alpha}$
such that $\tau\le\ind{{\jj}}{\beta}$, $\tau\le\ind{{\mm}}{\alpha}$,
we have
$$
  \dbinom{\ind{{\mm}}{\alpha}}{\ind{{\jj}}{\beta}} =
     \frac{\ind{{\mm}}{\alpha_s}}{\ind{{\jj}}{\beta_s}}\dbinom{\ind{{\mm}}{\alpha}-\tau}{\ind{{\jj}}{\beta}-\tau},
$$
where $\ind{{\mm}}{\alpha_s}$, $\ind{{\jj}}{\beta_s}$ are the $s$th
components of $\ind{{\mm}}{\alpha}$,
$\ind{{\jj}}{\beta}$, respectively.
For some $\ind{{\mm}'}{\alpha}$, $\ind{{\jj}'}{\beta}$ such that
$\tau\not\le\ind{{\mm}'}{\alpha}$ and $\tau\le\ind{{\jj}'}{\beta}$,
the binomial coefficients
$\dbinom{\ind{{\mm}'}{\alpha}-\tau}{\ind{{\jj}'}{\beta}-\tau}$ vanish; thus, by~\eqref{VectorInNullSpace},
for all $\ind{{\jj}}{\beta}\in\cA_{\lL}$  
such that $\tau\le\ind{{\jj}}{\beta}$,  we have
\begin{equation}\label{VectorInNullSpaceAgain}
  \sum_{\begin{subarray}{c}
          {\mm}\in\{1,\dots,\cc d(\LL)\},\\[1pt]
          \tau\le\ind{{\mm}}{\alpha}
        \end{subarray}}
   \frac{\ind{{\mm}}{\alpha_s}}{\ind{{\jj}}{\beta_s}}\dbinom{\ind{{\mm}}{\alpha}-\tau}{\ind{{\jj}}{\beta}-\tau}
        D^{\ind{{\mm}}{\alpha}-\tau-(\ind{{\jj}}{\beta}-\tau)}f_{\kk}(x_0)v_{\mm} = 0.
\end{equation}
In the previous formula, renumbering the indices ${\mm}$ and ${\jj}$
one after another, since a lexicographical order respects subtraction;
we can write
\begin{equation}\label{VectorInNullSpaceAgainII}
    \sum_{{\mm}=1}^{\cc d(\LL-1)}\dbinom{\ind{\sigma({\mm})}{\alpha}}{\ind{\sigma({\jj})}{\beta}}
        D^{\ind{\sigma({\mm})}{\alpha}-\ind{\sigma({\jj})}{\beta}}f_{\kk}(x_0)\te v_{{\mm}} = 0,\quad
  \begin{aligned}
       &\left(\ind{1}{\alpha},\dots,\ind{\cc d(\LL-1)}{\alpha}\right)=\cc\cA_{\LL-1},
           \\
       &\left(\ind{1}{\beta},\dots,\ind{d({\lL}-1)}{\beta}\right)=\cA_{{\lL}-1}.  
  \end{aligned}
\end{equation}
where
\begin{equation}\label{PermutationOfV}
  \te v_{{\mm}}:=\ind{\sigma({\mm})}{\alpha_s}v_{\sigma({\mm})}
\end{equation}
and $\sigma$s stand for the permutations of multi-indices.
(Note that $\sigma({\mm})$ and $\sigma({\jj})$ in
formula~\eqref{VectorInNullSpaceAgainII} are different
permutations.)

Suppose that the vector $v$ contains at least one non-zero component $v_t\ne0$, $t\in\{1,\dots,d(\LL)\}$.
In formula~\eqref{VectorInNullSpaceAgain}, taking some $\tau\in\Zp^d$, $|\tau|=1$, such that
$\tau\le\ind{t}{\alpha}$;
then sums in~\eqref{VectorInNullSpaceAgain} and~\eqref{VectorInNullSpaceAgainII}
contain the non-zero component $v_t$.
Now it is not hard to see that the vector $\te v\in\C^{d(\LL-1)}$ is a non-zero vector from the null-space of
$\bD^{{\lL}-1}_{\LL-1} f(x_0)$.
%
\end{proof}

\section{Proof of Theorem~\protect\ref{PolynomialSpasesInclusion}}\label{AppendixPolynomialSpasesInclusion}

\begin{proof}
Suppose a polynomial 
$p\in\cV_\LL$.
Then there exists a vector $v\in\ker\cc\bD_\LL P(x_0)$
such that $p=\left[\cc\cP_\LL\right]v$. Obviously,
\begin{equation}\label{InclusionToPi}
  p\in\Pi_{\le\LL}.
\end{equation}

The matrix $\cc\bD_{\LL+1}P(x_0)$ can be presented as a block matrix
\begin{equation}\label{BlockFormMatrixD}
  \cc\bD_{\LL+1}P(x_0):=
  \left[\begin{array}{ccc|c}
     &  && \bD^{0}_{\LL+1}P(x_0)\\
     & \raisebox{0.005\textheight}{$\cc\bD_{\LL}P(x_0)$} && \vdots\\
     &  && \bD^{\LL}_{\LL+1}P(x_0)\\
     \cline{1-3}
   0 & \dots & \multicolumn{1}{c}{0} & \bD^{\LL+1}_{\LL+1}P(x_0)
  \end{array}\right],
\end{equation}
where the submatrices $\bD^{0}_{\LL+1},\dots,\bD^{\LL+1}_{\LL+1}$ are given by~\eqref{MatrixDlL}.
Introduce an auxiliary column vector as follows $v^\sharp
:=\begin{bmatrix}
    v^T &
    \underbrace{\begin{array}{ccc} 0 & \cdots & 0\end{array}}_{d(\LL+1)}
  \end{bmatrix}^T$. 

Then, using block form~\eqref{BlockFormMatrixD} of the matrix $\cc\bD_{\LL+1}P(x_0)$,
we get $v^\sharp\in\ker\cc\bD_{\LL+1}P(x_0)$. Since $p=\left[\cc\cP_{\LL+1}\right]v^\sharp$, it follows that
$p\in\cV_{\LL+1}$; and, by\eqref{InclusionToPi}, we obtain $p\in\Pi_{\le \LL}\cap\cV_{\LL+1}$.
Thus we have $\cV_\LL\subseteq\Pi_{\le \LL}\cap\cV_{\LL+1}$.

Contrary. Suppose a polynomial $p\in\Pi_{\le \LL}\cap\cV_{\LL+1}$. Since $p\in\Pi_{\le\LL}$, it follows that
the polynomial can be presented as follows $p=\left[\cc\cP_{\LL+1}\right]v^\sharp$, 
$v^\sharp:=\begin{bmatrix}
    v^T &
    \underbrace{\begin{array}{ccc} 0 & \cdots & 0\end{array}}_{d(\LL+1)}
  \end{bmatrix}^T$, where $v\in\C^{\cc d(\LL)}$.  
Since $v^\sharp\in\ker\cc\bD_{\LL+1}P(x_0)$ and, arguing as above; 
$v\in\ker\cc\bD_{\LL}P(x_0)$. Thus $p\in\cV_\LL$; i.\,e., we have $\cV_\LL\supseteq\Pi_{\le \LL}\cap\cV_{\LL+1}$.

This concludes the proof.
\end{proof}

\end{document}